\documentclass{amsart}

\usepackage{amssymb}
\usepackage{amsfonts}
\usepackage{amsmath}
\usepackage{amsthm}
\usepackage{graphicx}
\usepackage{mathrsfs}
\usepackage{dsfont}
\usepackage{amscd}
\usepackage{multirow}
\usepackage{mathpazo}
\usepackage[all]{xy}
\normalfont
\usepackage[T1]{fontenc}
\usepackage{calligra}
\usepackage{verbatim}
\usepackage[usenames,dvipsnames]{color}
\usepackage[colorlinks=true,linkcolor=Blue,citecolor=Violet]{hyperref}

\makeatletter
\@namedef{subjclassname@2010}{%
  \textup{2010} Mathematics Subject Classification}
\makeatother

\theoremstyle{plain}
\newtheorem{theorem}{Theorem}[section]

\newtheorem{case}{Case}

\newtheorem{lemma}[theorem]{Lemma}
\newtheorem{proposition}[theorem]{Proposition}

\newtheorem{theorem-definition}[theorem]{Theorem-Definition}

\theoremstyle{definition}
\newtheorem{definition}[theorem]{Definition}

\newtheorem{notation}[theorem]{Notation}

\newtheorem{convention}[theorem]{Convention}

\theoremstyle{remark}

\newtheorem{example}[theorem]{Example}

\numberwithin{equation}{section}

\newcommand{\N}{{\mathds{N}}}

\newcommand{\R}{{\mathds{R}}}
\newcommand{\C}{{\mathds{C}}}

\newcommand{\D}{{\mathfrak{D}}}
\newcommand{\A}{{\mathfrak{A}}}
\newcommand{\B}{{\mathfrak{B}}}

\newcommand{\Lip}{{\mathsf{L}}}

\newcommand{\qpropinquity}[1]{{\mathsf{\Lambda}_{#1}}}

\newcommand{\StateSpace}{{\mathscr{S}}}

\newcommand{\Qqcms}[1]{${#1}$-quasi-Leibniz quantum compact metric space}

\newcommand{\sa}[1]{{\mathfrak{sa}\left({#1}\right)}}

\newcommand{\dom}[1]{{\operatorname*{dom}\left({#1}\right)}}

\renewcommand{\geq}{\geqslant}
\renewcommand{\leq}{\leqslant}

\newcommand{\Latremoliere}{Latr\'{e}moli\`{e}re}

\allowdisplaybreaks[1]

\makeatletter
\newcommand{\vast}{\bBigg@{4}}
\newcommand{\Vast}{\bBigg@{5}}
\makeatother

\newcommand{\myMod}[2]{#1 \ \mathrm{mod} \ #2}

\newcommand{\blockmx}[1]{
	\begin{bmatrix}
		#1 && \text{\huge 0}\\
		& \ddots & \\
		\text{\huge 0} && #1\\
	\end{bmatrix}}
\newcommand{\AthruB}[2]{\{#1, \cdots, #2 \}}
\newcommand{\floorfunc}[1]{\lfloor #1 \rfloor}

\begin{document}

\title[Quantum metrics from the trace on full matrix algebras]{Quantum metrics from the trace on full matrix algebras  }
\author{Konrad Aguilar$^1$}
\address{School of Mathematical and Statistical Sciences, Arizona State University, 901 S. Palm Walk, Tempe, AZ 85287-1804}
\email{konrad.aguilar@asu.edu}
\urladdr{}
\author{Samantha Brooker$^2$}
\address{Department of Mathematics, University of Denver, 2390 South York Street, Denver CO 80208}
\email{sjbrooke@asu.edu}
\urladdr{}

\date{\today\\ 
$^1$ Arizona State University \\
$^2$ University of Denver}
\thanks{The second author was partially supported by the {\em Summer Research Grant} provided by the Undergraduate Research Center at University of Denver}

\begin{abstract}
We prove that, in the sense of the Gromov-Hausdorff propinquity, certain natural quantum metrics on the algebras of $n\times n$-matrices are separated by a positive distance when n is not prime.

\end{abstract}
\maketitle

\setcounter{tocdepth}{1}
\tableofcontents

\section{Introduction and Background}
Motivated by high energy physics, M. A. Rieffel developed the notion of a "noncommutative" or "quantum" compact metric space, and initiated the study of topologies over classes of such quantum metric spaces \cite{Rieffel00, Rieffel98a, Rieffel01}.  The convergence is studied by use of a distance on the classes of these spaces.  F. \Latremoliere \ introduced the quantum Gromov-Hausdorff propinquity \cite{Latremoliere13, Latremoliere13c, Rieffel15} as a noncommutative analogue of the Gromov-Hausdorff distance \cite{burago01}, adapted to the theory of C*-algebras.

Quantum metric spaces are built by adding some quantum metric to unital C*-algebras \cite{Rieffel00} (see \cite{Latremoliere15b} for a survey with many examples and references). C*-algebras are certain norm-closed algebras of bounded linear operators on Hilbert spaces, up to an appropriate notion of *-isomorphism \cite[Theorem I.9.12]{Davidson}. 
Thus, it is natural to look to various classes of C*-algebras and study them from the viewpoint of quantum compact metric spaces and the Gromov-Hausdorff propinquity.  The first author and F. \Latremoliere \ did just that in \cite{AL} with a class of C*-algebras called approximately finite-dimensional C*-algebras (AF algebras)\cite{Bratteli72}.
Their work constructs quantum metrics on AF algebras, which is then used to study the topology induced by the propinquity on various natural sets of AF algebras. In particular, these quantum metrics can be restricted to the finite dimensional C*-subalgebras of AF algebras. The present work further studies some of the geometric aspects of this construction on these finite-dimensional algebras when it is the case that these finite-dimensional algebras are full matrix algebras (algebras of complex-valued square matrices). Our focus is to prove that different quantum metrics induced by the first author and \Latremoliere's construction on the same full matrix algebras are indeed at positive distance, in the sense of the propinquity. By its nature, it is usually difficult to prove lower bounds on the propinquity, so our work tackles a delicate aspect of this theory. First, we begin by defining many of these objects, from C*-algebras to quantum compact metric spaces, while taking note of some theorems that will prove useful to our efforts. Definitions (\ref{algebra-def}---\ref{states-def}) are contained in \cite[Chapter I]{Davidson}.
 \begin{definition}\label{algebra-def}
 An {\em associative algebra over the complex numbers } $\C$ is a vector space $\A$ over $\C$ with an associative multiplication, denoted by concatenation, such that:
 \begin{align*}
 & a(b+c)=ab+ac \text{ and } (b+c)a=ba+ca \text{ for all } a,b,c \in \A \\
 & \lambda (ab)=(\lambda a) b=a(\lambda b)  \text{ for all } a,b \in \A, \lambda \in \C.
 \end{align*}
 In other words, the associative multiplication is a bilinear map from $\A \times \A$ to $\A$.
 
 We say that $\A$ is {\em unital} if there exists a multiplicative identity, denoted by $1_\A$.  That is:
 \begin{equation*}
 1_\A a=a=a1_\A \text{ for all } a \in \A.
 \end{equation*}
 \end{definition}
 \begin{convention}
All algebras are associative algebras over the complex number $\C$ unless otherwise specified.
\end{convention}
\begin{notation}
When $E$ is a normed vector space, then its norm will be denoted by $\|\cdot\|_E$ by default.
\end{notation}
\begin{definition}\label{banach-algebra}
A {\em normed algebra} is an algebra $\A$ with a norm $\Vert \cdot \Vert_\A$ such that:
\begin{equation*}
\Vert ab \Vert_\A \leq \Vert a \Vert_\A \Vert b \Vert_\A \text{ for all } a,b \in \A.
\end{equation*}

$\A$ is a  {\em Banach Algebra}  when $\A$ is complete with respect to the norm $\Vert \cdot \Vert_\A$.
\end{definition}
\begin{definition}\label{c*-algebra}
A {\em C*-algebra}, $\A$, is a Banach algebra such that  there exists an anti-multiplicative conjugate linear involution ${}^* : \A \longrightarrow \A$, called the {\em adjoint}.  That is, * satisfies:
\begin{enumerate}
\item \text{{\em (conjugate linear):} } $(\lambda (a+b))^* = \overline{\lambda}(a^* + b^*) \text{ for all } \lambda \in \C, a,b \in \A;$  
\item \text{{\em (involution):} }  $(a^*)^*=a \text{ for all } a \in \A;$ 
\item  \text{{\em (anti-multiplicative):} }  $(ab)^* = b^*a^* \text{ for all } a,b \in \A.$ 
\end{enumerate}
Furthermore, the norm, multiplication, and  adjoint together satisfy the identity:
 \begin{equation}\Vert a a^*\Vert_\A= \Vert a \Vert_\A^2 \text{ for all } a \in \A
\end{equation}
 called the {\em C*-identity}.
 
 The set of self-adjoint elements of a C*-algebra is the set $\sa{\A}=\{a \in \A : a=a^*\}$.

An element $u \in \A$ of a unital C*-algebra is {\em unitary} if $uu^*=u^*u=1_\A$. 

We say that $\B \subseteq \A$ is a {\em C*-subalgebra of} $\A$ if $\B$ is a norm closed subalgebra that is also self-adjoint, i.e. $a \in \B \iff a^* \in \B$.  
\end{definition}

Our main example will be the C*-algebra of $n \times n$-matrices over the complex numbers, which we define now.

\begin{example}[{\cite[Example I.1.1]{Davidson}}]\label{full-matrix-ex}
Fix $n \in \N \setminus \{0\}$.  We let $M_n(\C)$ denote the the C*-algebra of $n \times n$-matrices over the complex numbers called a {\em full matrix algebra}.  The algebra is given by the standard matrix operations and the adjoint is the conjugate transpose. The norm is given by the operator norm: 
\begin{equation*}
\Vert a \Vert_{M_n(\C)} = \sup \left\{ \|ax\|_{2} : x\in \C^n, \|x\|_2 \leq 1 \right\} \text{ for all } a\in M_n(\C),
\end{equation*}
where $ax$ denotes matrix multiplication of the matrix $a$ and column vector $x$, and $\|\cdot \|_2$ is the standard Euclidean $2$-norm on $\C^n$.  

Note that if $a \in \sa{M_n(\C)}$, then $\Vert a \Vert_{M_n(\C)}=\max \{ |\lambda| : \lambda \text{ is an eigenvalue of } a\}$ by {\cite[Corollary I.3.4]{Davidson}}.  The unit of $M_n(\C)$ is the identity matrix, which we denote by $I_n$.  For $a \in M_n(\C)$, we denote the $i$-row, $j$-column entry by $a_{i,j}$.
\end{example}
Next, we describe morphisms between C*-algebras.
\begin{definition}\label{c*-isomorphism}
Let $\A, \B$ be C*-algebras.  A {\em *-homomorphism} $\pi : \A \longrightarrow \B$ is a *-preserving, linear and multiplicative function.  

$\pi$ is a {\em *-monomorphism} if it is an injective *-homomorphism.

$\pi$ is a { \em *-epimorphism} if  $\pi$ is a surjective *-homomorphism. 

$\pi$ is a { \em *-isomorphism} if  $\pi$ is a bijective *-homomorphism.

$\A$ is {\em *-isomorphic} to $\D$ if there exists a *-isomorphism $\pi : \A \longrightarrow \D$, and we then write $\A \cong \D$.

If both $\A,\D$ are unital, then we call a *-homomorphism $\pi: \A \longrightarrow \D$ {\em unital} if $\pi(1_\A)=1_\D .$ 
\end{definition}
The next result shows that there are important analytical properties  (such as continuity, contractibility, and isometry) associated to these morphisms without further assumptions. Thus, only algebraic conditions are indeed needed to properly define morphisms between C*-algebras in  Definition (\ref{c*-isomorphism}).
\begin{proposition}[{\cite[Theorem I.5.5]{Davidson}}]\label{morphism-prop}
Let $\A, \D$ be C*-algebras.  If $\pi : \A \longrightarrow \D$ is a *-homomorphism, then $\pi$ is  continuous and contractive.  That is, its operator norm: 
\begin{equation*}\Vert \pi \Vert_{\B(\A, \D)}= \sup \{ \Vert \pi(a)\Vert_\D : a \in \A, \Vert a \Vert_\A = 1 \} \leq 1,
\end{equation*} or equivalently, for all $a \in \A$, we have $\Vert \pi(a) \Vert_\D \leq \Vert a \Vert_\A$. 

If $\pi: \A \longrightarrow \D$ is a *-homomorphism, then $\pi$ is an isometry if and only if $\pi$ is a *-monomorphism.  In particular, *-isomorphisms are isometries.
\end{proposition}

In the next example, we present some *-isomorphisms related to the C*-algebras $M_n (\C)$.
\begin{example}[{\cite[Lemma III.2.1]{Davidson}}]\label{matrix-auto-example}
A map $\pi: M_n(\C) \longrightarrow M_n(\C)$ is a *-isomorphism if and only if there exists a unitary matrix $U \in M_n (\C)$ such that:
\begin{equation*}
\pi(a)=UaU^* \text{ for all } a \in M_n(\C).
\end{equation*}
\end{example}

In order to define quantum compact metric spaces we need to define another structure associated to C*-algebras.
\begin{definition}\label{states-def}
Let $\A$ be a unital C*-algebra. Let $\A'$ denote the set of continuous and linear complex-valued functions on $\A$. The {\em state space} of $\A$ is the set
\begin{equation*} \StateSpace(\A)=\left\{ \varphi \in \A' : 1=\varphi(1_\A) = \Vert \varphi \Vert_{\A'} \right\},
\end{equation*} 
where $\Vert \varphi \Vert_{\A'} = \sup \{\vert \varphi(a) \vert : a \in \A, \Vert a \Vert_\A =1\}$ is the operator norm.

A state $\varphi \in \StateSpace(\A)$ is called {\em tracial} if $\varphi(ab)=\varphi(ba)$ for all $a,b \in \A$ \cite[page 114]{Davidson}.
\end{definition}
As an example, we look to $M_n(\C)$.
\begin{example}\label{matrix-trace-example}
For $a \in M_n(\C)$, let $Tr_n(a)=\sum_{j =1}^n a_{j,j}$ be the trace of a matrix.  Define $tr_n = \frac{1}{n}Tr_n$.  By \cite[Example IV.5.4]{Davidson}, the map $tr_n$ is the unique tracial state on $M_n(\C)$.
\end{example}

In \cite{Rieffel98a}, M. A. Rieffel introduces the notion of a quantum compact metric space by providing a particular metric on the state space of a C*-algebra, which serves as a quantum analogue to the Monge-Kantorovich metric on Borel probablility measures of a compact Hausdorff space.  This metric lies outside the scope of this paper, so we provide an equivalent definition of a quantum compact metric space that utilizes a quantum analogue to the Lipschitz seminorm on continuous functions.

\begin{definition}[\cite{Rieffel98a,Rieffel99, Ozawa05}]\label{qcms-def}
Let $\A$ be a unital C*-algebra.  Let $\Lip$ be a seminorm on $\sa{\A}$ (possibly taking value $+\infty$).  The pair $(\A, \Lip)$ is a quantum compact metric space if $\Lip$ satisifies the following.
\begin{enumerate}
\item $\Lip$ is lower semi-continous with respect to $\Vert \cdot \Vert_\A$,
\item the set $\dom{\Lip}=\{ a \in \sa{\A}: \Lip(a) < \infty\}$ is dense in $\sa{\A}$,
\item the kernel of $\Lip$ is $\{a \in \sa{\A}: \Lip(a)=0\}=\R1_A =\{r1_\A \in \sa{\A} :  r \in \R\}$, and 
\item there exists a state $\mu \in \StateSpace(\A)$ such that the set:
\begin{equation*}
\{ a \in \sa{\A} : \mu(a)=0 \text{ and } \Lip(a) \leq 1\}
\end{equation*}
is totally bounded with respect to $\Vert \cdot \Vert_\A$. 
\end{enumerate}
\end{definition}
The Lipschitz seminorm on continuous functions satisfies a Leibniz rule and we may generalize this in the following way.
\begin{definition}
Fix $C \geq 1, D \geq 0$. A quantum compact metric space $(\A, \Lip)$ is a {\em {\Qqcms{(C,D)}}} if $\Lip$ is a \emph{$(C,D)$-quasi-Leibniz seminorm}, i.e. for all $a,b \in \sa{\A}$:
\begin{equation*}
\max\left\{ \Lip\left(a \circ b \right), \Lip\left(\{a,b\}\right) \right\} \leq C\left(\|a\|_\A \Lip(b) + \|b\|_\A\Lip(a)\right) + D \Lip(a)\Lip(b)\text{,}
\end{equation*}
where the Jordan product is $a \circ b=\frac{ab+ba}{2}$ and the Lie product is $\{a,b\}=\frac{ab-ba}{2i}$.
\end{definition}
In \cite{Latremoliere13}, F. \Latremoliere \ introduced a quantum analogue to the Gromov-Hausdorff distance \cite{burago01}, the quantum Gromov-Hausdorff propinquity. Indeed, the quantum Gromov-Hausdorff propinquity is a distance between quasi-Leibniz quantum compact metric spaces that preserves the C*-algebraic and metric structures and recovers the topology of the Gromov-Haus\-dorff distance, and thus provides an appropriate framework for the study of noncommutative metric geometry.  The definition is quite involved, so in the following theorem, we summarize  the results that pertain to our work in this paper, while also defining the standard notion of isomorphism between two quasi-Leibniz quantum comapct metric spaces, a quantum isometry.
\begin{theorem}[\cite{Latremoliere13, Latremoliere15}]\label{def-thm}
    The {\em quantum Gromov-Hausdorff propinquity} \\$\qpropinquity{}\left(\left(\A,\Lip_\A\right),\left(\B,\Lip_\B\right)\right)$ between two quasi-Leibniz quantum compact metric spaces $\left(\A,\Lip_\A \right)$ and $\left(\B,\Lip_\B \right)$ is a metric up to {\em quantum isometry}, i.e.  $\qpropinquity{}((\A,\Lip_\A),(\B,\Lip_\B)) = 0$ if and only if there exists a unital *-isomorphism $\pi : \A \rightarrow\B$ with $\Lip_\B\circ\pi = \Lip_\A$.

Moreover, $\qpropinquity{}$ recovers the Gromov-Hausdorff topology on compact metric spaces.
\end{theorem}

\section{Gromov-Hausdorff propinquity between isomorphic full matrix algebras}

In \cite{AL}, the first author and F. \Latremoliere \ discovered quasi-Lebiniz Lip-norms on certain infinite-dimensional C*-algebras called approximately finite-dimensional C*-algebras (AF algebras) with certain tracial states.  Now, these AF algebras are built by an inductive sequence of finite-dimensional AF algebras (see \cite[Chapter 6]{Murphy90}).  While it can be the case that two distinct inductive sequences can produce AF algebras that are *-isomorphic (see \cite[Example III.2.2]{Davidson}), the Lip-norms constructed in \cite[Theorem 3.5]{AL} seem to acknowledge the particular inductive sequence.  Therefore, the question arose of whether these Lip-norms can distinguish two *-isomorphic AF algebras with different inductive sequences, where by {\em distinguish,} we mean by way of a quantum isometry of Theorem  (\ref{def-thm}).  This would show that these Lip-norms are truly adding further structure to the C*-algebraic structure.  Yet, showing that two spaces are not  quantum isometric is quite a non-trivial task since the condition in Theorem (\ref{def-thm}) has to be checked for every *-isomorphism to provide a negative result. Indeed, two spaces $(\A, \Lip_\A), (\B, \Lip_\B)$ are {\em not} quantum isometric if and only if for any *-isomorphism $\pi: \A \longrightarrow \B$ (if it exists), we have that $\Lip_\B \circ \pi \neq \Lip_\A$.  But, in the case of $M_n(\C),$ we have that the  *-isomorphisms are well understood as seen in Example (\ref{matrix-auto-example}).  Thus, in this paper, we try to tackle this question of quantum isometry in the case of the finite-dimensional C*-algebras $M_n (\C)$ with respect to different finite inductive sequences, and we accomplish the task in Theorem (\ref{matrix-positive-propinquity}).  We begin by defining the particular quantum metric spaces that we will be working with, which requires the following notions.

\begin{definition}
A {\em conditional expectation}  $P : \A \longrightarrow \B$ onto $\B$, where $\A$ is a unital  C*-algebra and $\B$ is a unital C*-subalgebra of $\A$, is a linear map such that:
\begin{enumerate}
\item for all $a \in \A$ there exists $b \in \B$ such that $P(aa^*)=bb^*$, 
\item for all $ a\in \A$, we have $\Vert P(a) \Vert_\A \leq \Vert a \Vert_\A$,
\item for all $b,c \in \B$ and $a \in  \A$, we have $P(bac)=bP(a)c,$ and
\item $P(b)=b$ for all $b \in \B.$
\end{enumerate}
\end{definition}

\begin{notation}\label{matrix-isomorphisms}
We write $k |n$ to denote that $k$ divides $n$ throught this paper.

 Let $n,k \in \N \setminus \{0\}, n>1$ with $k |n$ and $k<n$.

 Let $\pi_{k,n}: M_k (\C) \longmapsto M_n (\C)$ be the unital *-monomorphism of \cite[Lemma III.2.1]{Davidson} determined by:
  \begin{equation*}
		\pi_{k,n}(a) = \blockmx{a} \text{ for all } a \in M_k(\C)
	\end{equation*}
	where there are $\frac{n}{k}$ non-overlapping copies of $a$ filling the block diagonal and $0$'s elsewhere.   Coordinate-wise, the map $\pi_{k,n}$ satisfies the following:
	\begin{equation*}
		\pi_{k,n}(a)_{p,q} =
			\begin{cases}
				a_{i,j} &: \text{if } \exists r \in \AthruB{0}{\frac{n}{k}-1} \text{ such that } p = i+rk \text{ and } q = j + rk \\
				 & \quad \text{ for some } i, j \in \AthruB{1}{k} \\
				0 &: \text{otherwise}
			\end{cases}
	\end{equation*}
	for all $a \in M_k(\mathbb{C})$ and $p,q \in \{1, \ldots n\}$.  
	
	Note that $\pi_{k,n}(M_k(\C))$ is a unital C*-subalgebra of $M_n(\C)$.  
\end{notation}
In Lemma (\ref{matrix-isomorphisms-lemma}), we will present another coordinate-wise description of $\pi_{k,n}$, which allows for direct computation rather than finding the terms and is much more suitable and convenient for the many calculations we will work with in this paper.  But, for now, we are prepared to  define the quantum metrics on $M_n(\C)$.
\begin{theorem}\label{matrix-qcms-theorem}
Let $n\in \N \setminus \{0,1\}$ such that there exists $k \in \N \setminus \{0,1\}$  with $k<n$ and $k | n$.  Recall the definition of the tracial state $tr_n : M_n (\C) \longrightarrow \C$ of Example (\ref{matrix-trace-example}).

If $P_{j,n}: M_n(\C) \longrightarrow \pi_{j,n}(M_j(\C))$ for $j=1,k$ denotes the unique $tr_n$-preserving condional expectation and we define for all $a \in \sa{M_n(\C)}$ the seminorms:
\begin{equation*}
L_{M_n(\C),1}(a)=\Vert a-P_{1,n}(a) \Vert_{M_n(\C)}
\end{equation*}
and 
\begin{equation*}
L_{M_n(\C),k}(a)= \max \left\{ \Vert a-P_{1,n}(a) \Vert_{M_n(\C)},k \cdot \Vert a-P_{k,n}(a) \Vert_{M_n(\C)}\right\},
\end{equation*}
then both $\left(M_n(\C), L_{M_n(\C),1}\right)$ and $\left(M_n(\C), L_{M_n(\C),k}\right)$ are $(2,0)$-quasi-Leibniz quantum compact metric spaces.
\end{theorem}
\begin{proof}
This is \cite[Theorem 3.5]{AL} and in particular Step 3 of its proof.
\end{proof}

Our main goal --- realized in Theorem (\ref{matrix-positive-propinquity}) --- is to show that the  two quantum metric spaces for a fixed $k,n$ are not quantum isometric.  These spaces are motivated by \cite[Theorem 4.9]{AL}, where these spaces do in fact form the finite-dimensional quanutm metric spaces used in their construction.  Therefore, our work is a legitimate  step towards understanding the quantum isometries between the infinite-dimensional C*-algebras presented in \cite{AL}.  Next, our proof of Theorem (\ref{matrix-positive-propinquity}) requires a detailed and coordinate-wise description of the maps $\pi_{k,n}$ and $P_{k,n}$, which could be easily implemented algorithmically.  We begin with $\pi_{k,n}$.

\begin{lemma}\label{matrix-isomorphisms-lemma}
If $k,n \in \N \setminus \{0\}, n>1$ such that $k |n$ and $k<n$, then
	the map $\pi_{k,n}: M_k(\C) \longrightarrow M_n(\C)$ of Notation (\ref{matrix-isomorphisms}) satisfies:
		\begin{equation*}
			\pi_{k,n}(a)_{p,q} =
				\begin{cases}
					a_{1+\myMod{(p-1)}{k}, 1+ \myMod{(q-1)}{k}} &: \text{if } \lfloor \frac{p-1}{k} \rfloor = \lfloor \frac{q-1}{k} \rfloor \\
					0 &: \text{otherwise}
				\end{cases}
		\end{equation*}
for all $a \in M_k (\C)$ and $p,q \in \{1, \ldots n\}$.
	\end{lemma}
	
		\begin{proof}
		Let $a \in M_k(\C)$ and fix	 $p, q \in \AthruB{1}{n}$. 
			\begin{case} Assume $\floorfunc{\frac{p-1}{k}}=  \floorfunc{\frac{q-1}{k}}$. 
			\end{case}
Now, since $k\vert n$, there exist $ r, s \in \AthruB{0}{\frac{n}{k}-1}$ and $i, j \in \AthruB{1}{k}$ such that $p = i + rk$ and $q = j + sk$. Therefore, we have $\floorfunc{\frac{p-1}{k}} = \floorfunc{\frac{i + rk - 1}{k}} = \floorfunc{\frac{i-1}{k} + r} = \floorfunc{\frac{i-1}{k}} + r = r$ because $0 \leq i-1 < k$.
			
Also, we have $\floorfunc{\frac{q-1}{k}} = \floorfunc{\frac{j + sk - 1}{k}} = \floorfunc{\frac{j-1}{k} + s} = \floorfunc{\frac{j-1}{k}} + s = s$ because $0 \leq j-1 < k$. Thus $r=s$ by the Case 1 assumption and $\pi_{k,n}(a)_{p,q}=a_{i,j}$ by Notation (\ref{matrix-isomorphisms}).  However, by modular arithmetic, since $c \mod d=c-d\floorfunc{\frac{c}{d}}$, we gather that $1+p-1-rk=i$ and $1+q-1-rk=j$ imply that:
\begin{equation*}  
			\pi_{k,n}(a)_{p,q} = a_{1 + \myMod{(p-1)}{k}, 1+ \myMod{(q-1)}{k}}.
			\end{equation*}

			\begin{case}  Assume $\floorfunc{\frac{p-1}{k}}\neq \floorfunc{\frac{q-1}{k}}$.
			\end{case}
			By Notation (\ref{matrix-isomorphisms}), assume by way of contradiction that there exists $r \in \{0, \ldots , \frac{n}{k}-1 \}$ such that $p=i+rk$ and $q=j+rk$ for some $i,j \in \{1, \ldots, k\}$. Then, the argument of Case 1 implies that $r \neq r$, a contradiction.  Hence  $\pi_{k,n}(a)_{p,q}=0$ by Notation (\ref{matrix-isomorphisms}).
\end{proof}
		
		Next, we move to understanding the conditional expectations $P_{k,n}$ in a coordinate-wise manner.  First, we provide some notation for a standard basis for $M_n(\C)$. 

\begin{notation}\label{matrix-units-notation}
		For $n \in \mathbb{N} \setminus \{0\}, j \in \AthruB{1}{n}, k \in \AthruB{1}{n}$, let  $E_{n,j,k} \in M_n(\mathbb{C})$ denote the standard {\em matrix unit} (see \cite[Section III.1]{Davidson}) defined coordinate-wise by:
\begin{equation*}
			(E_{n,j,k})_{p,q} =
				\begin{cases}
					1 &: p = j \text{ and } q = k \\
					0 &: \text{otherwise}\\
				\end{cases}
		\end{equation*}
		for $p, q \in \AthruB{1}{n}$, and note that the set $\{ E_{n,j,k} \in M_n(\C) : j, k \in \AthruB{1}{n} \}$ forms a basis for $M_n(\mathbb{C})$.  
		
		Assume that $l \in \N \setminus \{0\}$ and $l|n$.  Let
		 \begin{equation*}B_{l,n}= \{ \pi_{l,n}(a) \in M_n(\C) : a \text{ is a matrix unit of } M_l(\C) \}.
		 \end{equation*}
\end{notation}

Next, the tracial state $tr_n$ induces an inner product on $M_n(\C)$ via $\langle a,b\rangle= tr_n(b^*a)$.  In \cite{AL}, they use this observation and that the matrix units are orthogonal with respect to this inner product to provide a general description of $P_{k,n}$ in terms of the matrix units and $\pi_{k,n}$ (see \cite[Expression (4.1)]{AL}).  We will utilize this in Lemma (\ref{projection}) to provide an explicit coordinate-wise description of $P_{k,n}$.  But, first, we prove a lemma about the relationship between the tracial state $tr_n$  and the *-monomorphism $\pi_{k,n}$. 

\begin{lemma}\label{matrix-trace-lemma}
If $k, n \in \N \setminus \{0\}$ and  $k|n$, then  $tr_n \circ \pi_{k,n} = tr_k$.
\end{lemma}
\begin{proof}
Let $a \in M_k(\mathbb{C})$. Then, by Lemma (\ref{matrix-isomorphisms-lemma}):
					\begin{equation*}
						\begin{split}
							(tr_n \circ \pi_{k,n})(a) &= \frac{1}{n} Tr_n(\pi_{k,n}(a))\\
							&= \frac{1}{n} \sum_{i=1}^n \pi_{k,n}(a)_{i,i}\\
							&= \frac{1}{n} \sum_{i=1}^n a_{1 + \myMod{(i-1)}{k}, 1 + \myMod{(i-1)}{k}}\\
							&= \frac{1}{n} \sum_{i = 1}^{\frac{n}{k}} \sum_{j = 1}^k a_{j,j}\\
							&= \frac{1}{n} \sum_{i=1}^{\frac{n}{k}} Tr_k(a)\\
							&= \frac{1}{n} \cdot \frac{n}{k} Tr_k(a)\\
							&= \frac{1}{k} Tr_k(a)= tr_k(a),
						\end{split}
					\end{equation*}
					which completes the proof.
\end{proof}

\begin{lemma}\label{projection} If $n \in \N \setminus \{0\}$ and there exists $k \in \N \setminus \{0\}$ such that $k| n$, then  using notation from Theorem (\ref{matrix-qcms-theorem}), for all $a \in M_n (\mathbb{C})$, we have: 
		\begin{equation*}
			P_{k,n}(a) = \frac{k}{n} \sum_{p=1}^{k} \sum_{q=1}^{k} \left( \sum_{l=0}^{\frac{n}{k}-1} a_{p + kl, q + kl}\right) \pi_{k,n}(E_{k,p,q}),
		\end{equation*}
		and	coordinate-wise, this is:
		\begin{equation*}
			P_{k,n}(a)_{i,j} =
				\begin{cases}
					\frac{k}{n} \sum_{l=0}^{\frac{n}{k}-1} a_{p+kl, q+kl} &: \text{ if } \floorfunc{\frac{i-1}{k}} = \floorfunc{\frac{j-1}{k}}  \text{ and } p - 1 = \myMod{(i-1)}{k}\\
		&			\quad  \text{ and } q - 1 = \myMod{(j-1)}{k} \\
					0 &: \text{otherwise} 
				\end{cases}
		\end{equation*}
		for all $i,j \in \{1, \ldots, n\}$.

		Furthermore, if $k=1$, then $P_{1,n}(a)=tr_n(a)I_n$.
\end{lemma}  
\begin{proof}
By \cite[Expression (4.1)]{AL}, we have that:
		\begin{equation*}
			P_{k,n}(a) = \sum_{b \in B_{k,n}} \frac{tr_n(b^*a)}{tr_n(b^*b)} b.
		\end{equation*}
		for all $a \in M_n(\mathbb{C})$, where $B_{k,n}$ was defined in Notation (\ref{matrix-units-notation}). Thus, for  $b \in B_{k,n}$,  we have that $b = \pi_{k,n}(E_{k,p,q})$ for some $p,q \in \AthruB{1}{k}$, and thus 
				$b^*b = \pi_{k,n}(E_{k,q,p}) \pi_{k,n}(E_{k,p,q})= \pi_{k,n}(E_{k,q,p}  E_{k,p,q}).$ Now: 
					\begin{equation*}
			(E_{k,q,p} E_{k,p,q})_{i,j} = \sum_{l=1}^k (E_{k,q,p})_{i,l} \cdot (E_{k,p,q})_{l,j},
		\end{equation*}
		and:
		\begin{equation*}
			(E_{k,q,p})_{i,l} \cdot (E_{k,p,q})_{l,j} =
				\begin{cases}
					1 &: \text{if }i = q \text{ and } j = q \text{ and } l = p \\
					0 &: \text{otherwise}
				\end{cases}
		\end{equation*}
Therefore $E_{k,q,p}  E_{k,p,q} = E_{k,q,q}$. Hence, by Lemma (\ref{matrix-trace-lemma}), we gather that:
		\begin{equation*}
		\begin{split}
			tr_n(b^*b)&=tr_n(\pi_{k,n}(E_{k,q,p} E_{k,p,q}))\\
			 &= tr_k(E_{k,q,p} E_{k,p,q})\\
	&= tr_k(E_{k,q,q}) \\
				&= \frac{1}{k} Tr_k(E_{k,q,q})\\
				&= \frac{1}{k} \cdot 1 = \frac{1}{k}.
			\end{split}
		\end{equation*}

Thus, we now have:
		\begin{equation}\label{explicit-projection}
			\begin{split}
				P_{k,n}(a)& = \sum_{p = 1}^k \sum_{q=1}^k \frac{tr_n(\pi_{k,n}(E_{k,q,p})  a)}{\frac{1}{k}} \pi_{k,n}(E_{k,p,q})\\
				&= \sum_{p = 1}^k \sum_{q=1}^k k \cdot tr_n(\pi_{k,n}(E_{k,q,p})  a)  \pi_{k,n}(E_{k,p,q})\\
				&= \sum_{p = 1}^k \sum_{q=1}^k k \cdot \frac{1}{n} Tr(\pi_{k,n}(E_{k,q,p})  a)  \pi_{k,n}(E_{k,p,q})\\
				&= \frac{k}{n} \sum_{p = 1}^k \sum_{q=1}^k Tr(\pi_{k,n}(E_{k,q,p})  a) \pi_{k,n}(E_{k,p,q})\\
				&= \frac{k}{n} \sum_{p = 1}^k \sum_{q=1}^k \left( \sum_{i=1}^n (\pi_{k,n}(E_{k,q,p})  a)_{i,i} \right) \pi_{k,n}(E_{k,p,q}).
			\end{split}
		\end{equation}
		
		Next, fix $i \in \{1, \ldots, n \}$,by matrix multiplication:
		\begin{equation*}
			\sum_{i=1}^n (\pi_{k,n}(E_{k,q,p}) a)_{i,i} = \sum_{i=1}^n \sum_{j=1}^n \pi_{k,n}(E_{k,q,p})_{i,j} \cdot a_{j,i}
		\end{equation*}
		But, for $j \in \{1, \ldots, n\}$, we have:
	\begin{equation*}
			\pi_{k,n}(E_{k,q,p})_{i,j} =
				\begin{cases}
					1 &: \text{if }\exists l \in \AthruB{0}{\frac{n}{k}-1} \text{ such that } i = q +kl \text{ and } j = p + kl\\
					0 &: \text{otherwise}
				\end{cases}
	\end{equation*}
and:
		\begin{equation*}
			\pi_{k,n}(E_{k,q,p})_{i,j} \cdot a_{j,i} =
				\begin{cases}
					a_{j,i} &: \text{if }\exists l \in \AthruB{0}{\frac{n}{k}-1} \text{ such that } j = p+kl \text{ and } i = q + kl\\
					0 &: \text{otherwise}
				\end{cases}
		\end{equation*}
Hence: 
	\begin{equation*}
				\sum_{i=1}^n \sum_{j=1}^n \pi(E_{k,q,p})_{i,j} \cdot a_{j,i} = \sum_{l=0}^{\frac{n}{k}-1} a_{p + kl, q + kl}.
			\end{equation*}
		And, by Expression (\ref{explicit-projection}), we conclude that:
		\begin{equation}\label{proj-expression}
			P_{k,n}(a) = \frac{k}{n} \sum_{p=1}^k \sum_{q=1}^k \left( \sum_{l=0}^{\frac{n}{k}-1} a_{p + kl, q + kl} \right)\pi_{k,n}( E_{k,p,q}).
		\end{equation}
		
		The coordinate-wise expression follows from Lemma (\ref{matrix-isomorphisms-lemma}). Indeed:
		
		Let $i,j \in \{1, \ldots, n\}$.  If $\floorfunc{ \frac{i - 1}{k}} = \floorfunc{ \frac{j-1}{k}}$, then $P_{k,n}(a)_{i,j}$ lies in one of the $k$-by-$k$ diagonal blocks of the $n$-by-$n$ matrix. So, there exist $p, q$ such that $\pi_{k,n}(E_{k,p,q})_{i,j} = 1$. Namely $p = 1 + \myMod{(i-1)}{k}, q = 1 + \myMod{(j-1)}{k}$. However, this pair $p, q$ corresponds to a term in the sum definining $P_{k,n}$ in  Expression (\ref{proj-expression}), or corresponds to $\frac{k}{n} \left( \sum_{l=0}^{\frac{n}{k} - 1} a_{p + kl, q + kl} \right) \pi_{k,n}(E_{k,p,q})$. And, since the $i,j$th entry of $\pi_{k,n}(E_{k,p,q})$ is 1, this gives us that:
		\begin{equation*}
			P_{k,n}(a)_{i,j}=\sum_{l=0}^{\frac{n}{k} - 1} a_{p + kl, q+kl}.
		\end{equation*}
		
	If $\floorfunc{\frac{i-1}{k}} \neq \floorfunc{\frac{j-1}{k}}$, then for all $p, q$ between 1 and $k$, $(\pi_{k,n}(E_{k,p,q}))_{i,j} = 0$, by definition of $\pi_{k,n}$, so $(P_{k,n}(a))_{i,j} = 0$. 
	
	The last statement of this Lemma follows from this coordinate-wise description with $k=1$.
\end{proof}

For computational purposes, we present an alternative perspective on the projection map before proceeding.

\begin{proposition}\label{projection-explained}
	Let $a \in M_n(\mathbb{C})$, and let $k$ be an integer that divides $n$. Consider the $k$-by-$k$ diagonal blocks of $a$, denoted from top left diagonal block to bottom right diagonal block by $B_1, B_2, \dots, B_{n/k}$. The projection $P_{k,n}(a)$ is the image of the arithmetic mean of these blocks under the map $\pi_{k,n}$. In other words,
	
	\begin{equation*}
		P_{k,n}(a) = \pi_{k,n} \left( \frac{k}{n} \cdot \sum_{i=1}^{n/k} B_i \right).
	\end{equation*}
\end{proposition}

\begin{proof}
	This follows from Lemma (\ref{projection}).
\end{proof}

Finally, we are in a position to study the Lip-norms of Theorem (\ref{matrix-qcms-theorem}).

\begin{lemma}\label{unitary-vanish}
Let $n \in \N \setminus \{0,1\}$. Using notation from Theorem (\ref{matrix-qcms-theorem}), if 
		$\pi : M_n(\mathbb{C}) \rightarrow M_n(\mathbb{C})$ is a *-isomorphism, then:	
		\begin{equation*}
			L_{M_n(\mathbb{C}),1} \circ \pi(a) = L_{M_n(\mathbb{C}),1}(a).
		\end{equation*}
		\end{lemma}
		\begin{proof}
			Let $\pi: M_n (\C) \longrightarrow M_n(\C)$ be a *-isomorphism.  By \cite[Lemma III.2.1]{Davidson}, there exists a unitary $U \in M_n(\mathbb{C})$ such that $\pi(a)=UaU^*$ for all $a \in M_n(\C)$. Also, by Lemma (\ref{projection}), we gather $P_{1,n}(UaU^*) = tr_n(UaU^*)  I_n=tr_n(U^*Ua)I_n= tr_n(a) I_n$  for all $a \in M_n(\C)$.  Now, let $a \in M_n (\mathbb{C})$, thus:
\begin{equation*}
\begin{split}
L_{M_n(\mathbb{C}),1} \circ \pi(a) &= \Vert UaU^* - P_{1,n}(UaU^*)\Vert_{M_n(\mathbb{C})}\\
&=	\Vert UaU^* - P_{1,n}(UaU^*)\Vert_{M_n(\mathbb{C})}\\
& = \Vert UaU^* - tr_n(a) I_n\Vert_{M_n(\mathbb{C})}\\
&=\Vert UaU^* -  U(tr_n(a)  I_n)U^*\Vert_{M_n(\mathbb{C})}\\
&=\Vert U(a - tr_n(a) I_n)U^*\Vert_{M_n(\mathbb{C})}\\
&= \Vert U(a - P_{1,n}(a))U^*\Vert_{M_n(\mathbb{C})}\\
				&= \Vert a - P_{1,n}(a)\Vert_{M_n(\mathbb{C})} \\
				&=L_{M_n(\mathbb{C}),1} (a),
\end{split}
		\end{equation*}
		which completes the proof.
		\end{proof}
		\begin{lemma}\label{l-disagree-general}
If $k, n \in \mathbb{N} \backslash \{0\}, n>2$  such that $k \vert n$ and $1<k<n$, then using notation from Theorem (\ref{matrix-qcms-theorem}), we have:
\begin{equation*}L_{M_n(\mathbb{C}),1} \neq L_{M_n(\mathbb{C}),k}.
\end{equation*}
\end{lemma}
\begin{proof}
	Consider $a \in M_n(\mathbb{C})$:
	\begin{equation*}
		a = \begin{bmatrix}
			k & & & \text{\huge 0}\\
			 & 0 && \\
			&& \ddots &\\
			\text{\huge 0}&&& 0\\
		\end{bmatrix}
	\end{equation*}
	Now, by Lemma (\ref{projection}), we have:
	\begin{equation*}
		P_{1,n}(a) = \blockmx{\frac{k}{n}},
	\end{equation*}
	and thus:
	\begin{equation*}
		a - P_{1,n}(a) =
			\begin{bmatrix}
				\frac{k(n-1)}{n} &&&& \text{\huge 0}\\
				& -\frac{k}{n} &&& \\
				&& -\frac{k}{n} && \\
				&&& \ddots & \\
				\text{\huge 0} &&&& -\frac{k}{n}
			\end{bmatrix} \in \sa{M_n(\C)}
	\end{equation*}
	Therefore by Example (\ref{full-matrix-ex}), we have 
		$L_{M_n(\mathbb{C}),1}(a) =\Vert a - P_{1,n}(a) \Vert_{M_n(\mathbb{C})} = \frac{k(n-1)}{n}$ 
	since $\frac{k}{n}(n-1) \geq \frac{k}{n}.$
	
	Now, we consider $P_{k,n}$.  By Proposition (\ref{projection-explained}), we need only examine the diagonal $k$-by-$k$ blocks of $a$, which are:
	\begin{equation*}
		\begin{bmatrix}
			k & 0 &\cdots & 0 \\
			0 &0 & \cdots & 0\\
			\vdots & \vdots & \ddots  & \vdots \\
			0 & 0 & \cdots & 0\\
		\end{bmatrix}
		\begin{bmatrix}
			0 & 0 & \cdots & 0\\
			\vdots & \vdots & \cdots  & \vdots \\
			\vdots & \vdots & \cdots  & \vdots \\
			0 & 0 & \cdots & 0\\
		\end{bmatrix}
		\cdots
		\begin{bmatrix}
			0 & 0 & \cdots & 0\\
			\vdots & \vdots & \cdots  & \vdots \\
			\vdots & \vdots & \cdots  & \vdots \\
			0 & 0 & \cdots & 0\\
		\end{bmatrix}.
	\end{equation*}
	Summing these $\frac{n}{k}$ matrices, we get:
	\begin{equation*}
		\begin{bmatrix}
			k & 0 &\cdots & 0 \\
			0 &0 & \cdots & 0\\
			\vdots & \vdots & \ddots  & \vdots \\
			0 & 0 & \cdots & 0\\
		\end{bmatrix},
	\end{equation*}
	and when we divide by $\frac{n}{k}$ and take the image under $\pi_{k,n}$, we arrive at:
	\begin{equation*}
		P_{k,n}(a) =
			\pi_{k,n}
			 \left(	\begin{bmatrix}
					\frac{k^2}{n} &&& \text{\huge 0}\\
					& 0 &&\\
					&& \ddots &\\
					\text{\huge 0} &&& 0\\
				\end{bmatrix} \right)
	\end{equation*}
	by Proposition (\ref{projection-explained}).
	Thus:
	\begin{equation*}
		a - P_{k,n}(a) =
			\begin{bmatrix}
				\frac{k(n-k)}{n} &&&&&&&&&  \text{\huge 0}\\
				& \ddots &&&&&&&& \\
				&& 0 &&&&&&& \\
				&&& -\frac{k^2}{n} &&&&&&\\
				&&&& \ddots &&&&&\\
				&&&&& 0 &&&&\\
				&&&&&& \ddots &&& \\
				&&&&&&& -\frac{k^2}{n}  &&\\
				&&&&&&&& \ddots &\\
				\text{\huge 0} &&&&&&&&& 0
			\end{bmatrix} \in \sa{M_n(\C)}
	\end{equation*}
	by definition of $\pi_{k,n}$ in Notation (\ref{matrix-isomorphisms}). 
	Since $k \vert n$ and $k<n,$ we have that $k \leq \frac{n}{2}$. This implies that $n - k \geq k$, and so $k(n-k) \geq k^2$. Therefore $\frac{k(n-k)}{n} \geq \frac{k^2}{n}$, which implies that:
	\begin{equation*}
		\Vert a - P_{k,n}(a) \Vert_{M_n(\mathbb{C})} = \frac{k(n-k)}{n} \text{ and }		k \cdot \Vert a - P_{k,n}(a) \Vert_{M_n(\mathbb{C})} = \frac{k^2 (n-k)}{n}
	\end{equation*}
	by Example (\ref{full-matrix-ex}). 
Next, since $k \leq \frac{n}{2}$, we have $n-k \geq \frac{n}{2}$ as $k \geq 2.$ Hence $k(n-k) \geq 2 \cdot \frac{n}{2} = n \implies
	 k(n-k) \geq n > n - 1$, which means $k^2(n-k) > k(n-1)$, and finally we have:
	\begin{equation*}
		\frac{k^2(n-k)}{n} > \frac{k(n-1)}{n},
	\end{equation*}
and thus, $k \cdot \Vert a - P_{k,n}(a) \Vert_{M_n(\mathbb{C})} > \Vert a - P_{1,n}(a) \Vert_{M_n(\mathbb{C})}$.  Therefore:
	\begin{equation*}
		L_{M_n(\mathbb{C}),k}(a) = \frac{k^2(n-k)}{n}
	\end{equation*}
	and we conclude that $L_{M_n(\mathbb{C}),1} \neq L_{M_n(\mathbb{C}),k}$.
\end{proof}

\begin{theorem}\label{matrix-positive-propinquity} Using notation from Lemma (\ref{l-disagree-general}), if $k, n \in \mathbb{N} \backslash \{0\}, n>2$  such that $k \vert n$ and $1<k<n$, then the quantum compact metric spaces:
	\begin{equation*}
		\left(M_n (\mathbb{C}), L_{M_n(\mathbb{C}),1}\right) \text{ and } \left(M_n (\mathbb{C}), L_{M_n(\mathbb{C}),k}\right)
	\end{equation*} 
are not quantum isometric of Theorem (\ref{def-thm}), and therefore:
	\begin{equation*}
		\mathsf{\Lambda} \left(\left(M_n (\mathbb{C}), L_{M_n(\mathbb{C}),1}\right) , \left(M_n (\mathbb{C}), L_{M_n(\mathbb{C}),k}\right) \right) >0,
	\end{equation*}
where $\mathsf{\Lambda}$ is the quantum Gromov-Hausdorff propinquity of Theorem (\ref{def-thm}).
\end{theorem}

\begin{proof}

	By Theorem (\ref{def-thm}), we show for all unital *-isomorphisms $\pi:M_n (\mathbb{C}) \longrightarrow M_n(\C)$  that $L_{M_n(\mathbb{C}),1} \circ \pi \neq L_{M_n(\mathbb{C}),k}$. Let $\pi: M_n (\mathbb{C}) \rightarrow M_n (\mathbb{C})$, be a unital *-isomorphism. Therefore $L_{M_n(\mathbb{C}),1} \circ \pi = L_{M_n(\mathbb{C}),1}$ by Lemma (\ref{unitary-vanish}). But, Lemma (\ref{l-disagree-general}) implies that $L_{M_n(\mathbb{C}),1} \neq L_{M_n(\mathbb{C}),k}$. Thus, we conclude that $L_{M_n(\mathbb{C}),1} \circ \pi \neq L_{M_n(\mathbb{C}),k}$, and therefore
		$\left(M_n(\mathbb{C}), L_{M_n(\mathbb{C}),1}\right)$  and $\left(M_n (\mathbb{C}), L_{M_n(\mathbb{C}),k}\right)
$ 
are not quantum isometric.
\end{proof}

\section*{Acknowledgements}

The authors are grateful to Dr. Fr\'{e}d\'{e}ric Latr\'{e}moli\`{e}re for his feedback and support.

\vfill

\end{document}